\documentclass[12pt,a4paper]{article}

\usepackage{fontspec}
\usepackage[english]{babel}
\usepackage{amsmath,amssymb,amsthm,mathtools}
\numberwithin{equation}{section}
\usepackage{geometry}
\usepackage{booktabs,array}
\usepackage{xcolor}
\usepackage{csquotes}
\usepackage{microtype}
\usepackage{tikz}
\usetikzlibrary{arrows.meta,positioning,calc}
\usepackage[
    backend=biber,
    style=numeric,
    sorting=nyt,
    sortcites=true,
    giveninits=true,
    maxbibnames=99
]{biblatex}
\DeclareNameAlias{author}{given-family}
\DeclareNameAlias{editor}{given-family}
\DeclareFieldFormat[article]{title}{#1}
\DeclareFieldFormat[article]{number}{\mkbibparens{#1}}
\renewbibmacro*{volume+number+eid}{%
    \printfield{volume}%
    \setunit*{\addspace}%
    \printfield{number}%
    \setunit{\bibeidpunct}%
    \printfield{eid}}
\renewbibmacro{in:}{}
\defbibheading{bibliography}[\refname]{\section*{\centering #1}}
\AtBeginBibliography{\sloppy}
\usepackage{hyperref}
\hypersetup{
    pdftitle={Hausdorff Dimension of the Set of Extreme Points of a Random Countable Stable Zonotope},
    pdfauthor={Maksim Kukushkin},
    pdfsubject={Exact dimension of the set of extreme points of a random countable stable zonotope},
    pdfkeywords={random zonotope, extreme point, Hausdorff dimension, stable random vector, Frostman method}
}
\newif\ifprintversion
\ifprintversion
    \hypersetup{hidelinks}
\else
    \hypersetup{
        colorlinks=true,
        linkcolor=blue!55!black,
        citecolor=blue!55!black,
        urlcolor=blue!55!black
    }
\fi

\newtheorem{theorem}{Theorem}[section]
\newtheorem{lemma}[theorem]{Lemma}
\newtheorem{proposition}[theorem]{Proposition}
\newtheorem{corollary}[theorem]{Corollary}
\theoremstyle{definition}
\newtheorem{definition}{Definition}[section]
\newtheorem{remark}{Remark}[section]
\theoremstyle{plain}

\newcommand{\R}{\mathbb R}
\newcommand{\E}{\mathbb E}
\newcommand{\1}{\mathbf 1}
\newcommand{\Sph}{S^{d-1}}
\newcommand{\dd}{\mathop{}\!\mathrm d}
\newcommand{\eps}{\varepsilon}
\newcommand*{\eq}{\mathrel{\overset{\mathrm{def}}{=}}}
\newcommand{\step}[1]{\par\medskip\noindent\textbf{#1}\par\smallskip}
\newcommand{\nbcite}[2][]{\mbox{\cite[#1]{#2}}}
\DeclareMathOperator{\ext}{ext}
\DeclareMathOperator{\expo}{expo}
\DeclareMathOperator{\diam}{diam}
\DeclareMathOperator{\dimH}{dim_H}
\let\leq\leqslant
\let\geq\geqslant

\title{Hausdorff Dimension of the Set of Extreme Points\\
of a Random Countable Stable Zonotope}
\author{Maksim Kukushkin\\[-0.1em]
\small Faculty of Mathematics and Computer Science\\[-0.15em]
\small Saint Petersburg State University, Saint Petersburg, Russia}
\date{}

\begin{document}
\maketitle

\begin{abstract}
Let \(d\geq2\), \(0<\alpha<1\), and let \(\Gamma_k\) be the successive arrival
times of a standard Poisson process on \((0,\infty)\). Given independent
uniform directions \(\eps_k\in\Sph\), independent of \((\Gamma_k)\), we
consider the random countable stable zonotope
\[
    Z_\alpha=\bigoplus_{k=1}^{\infty}
    \Gamma_k^{-1/\alpha}[0,\eps_k].
\]
For the set of extreme points \(\ext Z_\alpha\), we prove that almost surely
\[
    \dimH\ext Z_\alpha=(d-1)\alpha,
\]
and that the critical Hausdorff measure
\(\mathcal H^{(d-1)\alpha}(\ext Z_\alpha)\) is almost surely finite. The lower
bound follows from the tangential non-degeneracy of the stable increments of
the parametrizing field and Frostman's energy criterion. For the upper bound
we construct an adaptive covering: at each scale the Poisson jumps are split
into large and small ones, the large jumps determine a finite hyperplane
arrangement, and the sum of the small jumps controls the diameters of the
images of its cells.
\end{abstract}

\medskip
\begin{center}
\begin{minipage}{0.88\textwidth}\small
\noindent\textbf{Keywords:} random countable zonotope, extreme point,
Hausdorff dimension, LePage series, stable random vector, Frostman method.
\end{minipage}
\end{center}
\section{Introduction}\label{sec:introduction}

Random countable zonotopes arise as stable random elements of the cone of
compact convex sets endowed with Minkowski addition. Gin{\'e} and
Hahn~\nbcite{gine1985} proved that, for \(1\leq\alpha\leq2\), every
\(\alpha\)-stable random compact convex set can be represented as
\(K\oplus\{\xi\}\), where \(K\) is a deterministic compact convex set and
\(\xi\) is an \(\alpha\)-stable random vector: the shape of the set is
deterministic and all randomness is confined to a translation. Nontrivial
random geometry therefore occurs only for \(0<\alpha<1\), which is the range
considered below. The general probabilistic theory of stable laws on cones was
developed in~\nbcite{davydov2000,davydov2004,davydov2008}; a systematic account
of random set theory is given in~\nbcite{molchanov2005}. The geometry of the set
of extreme points was isolated as a separate problem in~\nbcite{davydov2019}.

In dimension \(d=2\), Davydov and Paulauskas~\nbcite[Theorem~5]{davydov2019}
proved the upper bound \(\dimH\ext Z_\alpha\leq\alpha\) for every spectral
distribution on the circle. Under the additional
condition~\nbcite[Theorem~7]{davydov2019} that the angle law be the image of the
uniform law on \([0,\pi]\) under a bimeasurable bijection, they proved the
reverse bound and hence equality; the uniform measure meets this condition.
Their proof unfolds one
half of the boundary of a planar zonotope into the range of a one-dimensional
\(\alpha\)-stable subordinator and does not extend directly to dimensions
\(d\geq3\).

For \(d\geq3\), the standard projection argument yields only
\(\dimH\ext Z_\alpha\geq\alpha\). Indeed, a projection of a zonotope is again
a zonotope, \(\ext(\pi K)\subset\pi(\ext K)\) for a linear map \(\pi\), and a
Lipschitz image cannot increase Hausdorff dimension; thus the dimension of a
two-dimensional projection is a lower bound for that of the original set. The
resulting value \(\alpha\) is independent of \(d\) and does not reflect the
dimension of the ambient space. It was observed in
\nbcite[Remark~8]{davydov2019} that the projection method cannot provide an
upper bound, and it was conjectured that the correct value is
\((d-1)\alpha\); the case \(d=3\) was identified there as the first step
toward this conjecture.

A first upper bound in arbitrary dimension was obtained by the author
in~\nbcite{kukushkin2025}:
\[
    \dimH\ext Z_\alpha\leq\alpha+d-2
    \qquad\text{almost surely}.
\]
For \(d\geq3\), this estimate is strictly larger than the conjectured value
\((d-1)\alpha\) and therefore does not close the gap between the known lower
and upper bounds.

The present paper closes this gap for the uniform spectral measure by proving
the conjecture. Both estimates are built from the random field
\[
    X(u)=\sum_{k=1}^{\infty}\Gamma_k^{-1/\alpha}\eps_k
    \1_{\{\langle u,\eps_k\rangle>0\}},
    \qquad u\in\Sph,
\]
which parametrizes the exposed points of the zonotope, although the two halves
of the proof use substantially different methods. In addition to the exact
dimension, the proof of the upper bound yields finiteness of the critical
Hausdorff measure.

\subsection{Notation and conventions}\label{sec:notation}

Throughout, \(d\geq2\) and \(N=d-1\). Constants \(c,C\) may change from line
to line and depend only on explicitly fixed parameters, most notably on \(d\)
and \(\alpha\). Symbols such as \(C_N,C_q,C_\beta\) denote constants that may
also depend on the indicated subscript. Differentials are set in roman type,
as in \(\dd x\), and \(\eq\) denotes equality by definition.

The principal notation is summarized below.

\begin{center}
\small
\begin{tabular}{@{}p{34mm}p{105mm}@{}}
\toprule
notation & meaning\\
\midrule
\(d\),\ \(N=d-1\) & dimension of the ambient space and of the parameter sphere\\
\(\alpha\in(0,1)\) & stability index\\
\(\Gamma_k\),\ \(\gamma_k\eq\Gamma_k^{-1/\alpha}\) & arrival times of a standard Poisson process and segment weights\\
\(\eps_k\),\ \(\sigma\) & random directions and their common law, the uniform probability measure on \(\Sph\)\\
\(\lambda\) & restriction of \(\sigma\) to a geodesic ball \(U\subset\Sph\)\\
\(d_S\),\ \(\theta\) & geodesic distance on the sphere and the angle between directions\\
\(Z_\alpha\),\ \(F_\alpha\eq\ext Z_\alpha\) & the zonotope and the set of its extreme points\\
\(X(u)\),\ \(G\) & parametrizing field and the set of good directions\\
\(\Pi\),\ \(\nu\) & Poisson random measure of the jumps and its intensity\\
\(\Delta(u,v)\),\ \(Y(u,v)\),\ \(P_u\) & field increment, its tangential projection, and the orthogonal projection onto \(u^\perp\)\\
\(\Psi_{u,v}\),\ \(\Psi_\theta\eq\Psi_{e_d,v_\theta}\) & characteristic exponent of an increment and its canonical form\\
\(h=2^{-n}\),\ \(\delta\eq h^{1/\alpha}\) & cell scale and threshold separating large and small jumps\\
\(\mathcal Q_n\),\ \(Q\) & covering of the sphere by cells at scale \(h\), and one such cell\\
\(A(Q)\),\ \(p_Q\) & directions whose great spheres meet \(Q\), and the measure of this set\\
\(M_Q\),\ \(V_Q\) & number of relevant large jumps and total variation of the corresponding small jumps\\
\(\mathcal H^s\),\ \(\dimH\) & Hausdorff measure and Hausdorff dimension\\
\(I_s(\mu)\) & \(s\)-energy of the measure \(\mu\)\\
\bottomrule
\end{tabular}
\end{center}

\subsection{Structure of the paper}\label{sec:structure}

The logical dependencies are shown in Fig.~\ref{fig:proof-map}.
Section~\ref{sec:setup} states the main result and outlines the proof.
Section~\ref{sec:parametrization} relates the extreme points to the image of
the field \(X\) and introduces its Poisson representation. The lower and upper
bounds are proved in Sections~\ref{sec:lower} and~\ref{sec:upper},
respectively. Section~\ref{sec:completion} combines the bounds and proves
finiteness of the critical Hausdorff measure, while
Section~\ref{sec:discussion} discusses the role of the uniform spectral
measure and open problems. Appendix~\ref{app:technical} contains a geometric
explanation of the tangential projection and two computations omitted from the
main text; Appendix~\ref{app:standard} records the standard facts used in the
proof.

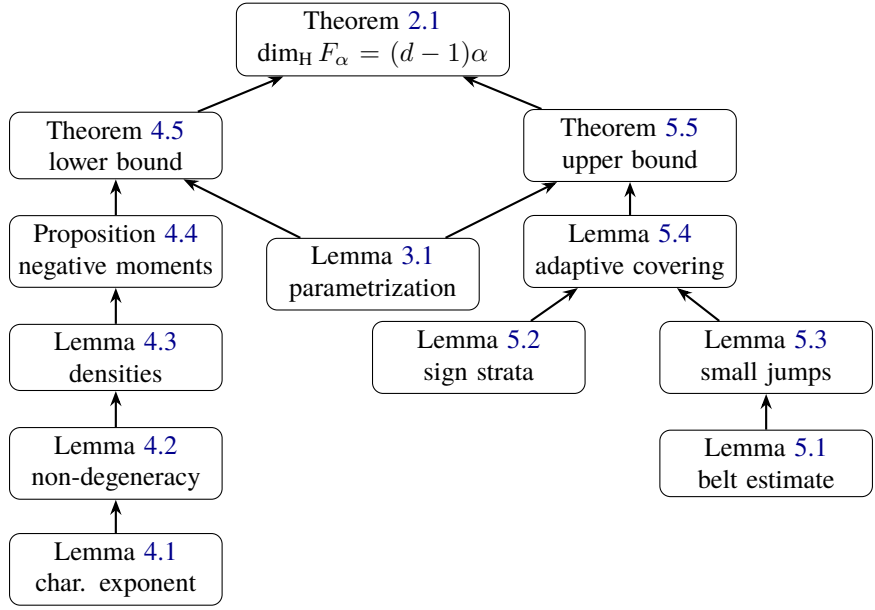
\begin{figure}[ht]
\centering
\begin{tikzpicture}[
    x=1cm,y=1cm,
    box/.style={draw,rounded corners,align=center,inner sep=3pt,
                text width=26mm,font=\footnotesize},
    wide/.style={box,text width=34mm},
    arr/.style={-{Stealth[length=2mm]},thick}
]
\node[wide] (main) at (0,7.4) {Theorem~\ref{thm:main}\\\(\dimH F_\alpha=(d-1)\alpha\)};
\node[box] (low)  at (-3.4,6.0) {Theorem~\ref{thm:lower}\\lower bound};
\node[box] (up)   at (3.4,6.0)  {Theorem~\ref{thm:upper}\\upper bound};
\node[box] (par)  at (0,4.3)    {Lemma~\ref{lem:parametrization}\\parametrization};
\node[box] (neg)  at (-3.4,4.6) {Proposition~\ref{prop:negative-moments}\\negative moments};
\node[box] (cov)  at (3.4,4.6)  {Lemma~\ref{lem:adaptive-cover}\\adaptive covering};
\node[box] (den)  at (-3.4,3.2) {Lemma~\ref{lem:bounded-densities}\\densities};
\node[box] (arr2) at (1.4,3.2)  {Lemma~\ref{lem:arrangement}\\sign strata};
\node[box] (sml)  at (5.2,3.2)  {Lemma~\ref{lem:small-jumps}\\small jumps};
\node[box] (tan)  at (-3.4,1.8) {Lemma~\ref{lem:tangent}\\non-degeneracy};
\node[box] (blt)  at (5.2,1.8)  {Lemma~\ref{lem:belt}\\belt estimate};
\node[box] (chr)  at (-3.4,0.4) {Lemma~\ref{lem:characteristic}\\char. exponent};
\draw[arr] (low)--(main);
\draw[arr] (up)--(main);
\draw[arr] (par)--(low);
\draw[arr] (par)--(up);
\draw[arr] (neg)--(low);
\draw[arr] (cov)--(up);
\draw[arr] (den)--(neg);
\draw[arr] (tan)--(den);
\draw[arr] (chr)--(tan);
\draw[arr] (arr2)--(cov);
\draw[arr] (sml)--(cov);
\draw[arr] (blt)--(sml);
\end{tikzpicture}
\caption{Logical structure of the proof. Each arrow points from a result to
the statement that uses it.}\label{fig:proof-map}
\end{figure}

\section{Problem setting and main result}\label{sec:setup}

\subsection{Random countable stable zonotope}

Let \(\tau_1,\tau_2,\ldots\) be independent standard exponential random
variables, let \(\Gamma_k=\tau_1+\cdots+\tau_k\), and let
\(\eps_1,\eps_2,\ldots\) be independent uniform points of \(\Sph\),
independent of the sequence \((\tau_k)\). Set
\(\gamma_k\eq\Gamma_k^{-1/\alpha}\). By the strong law of large numbers,
\(\Gamma_k/k\to1\) almost surely. Hence \(\gamma_k\sim k^{-1/\alpha}\), and,
since \(1/\alpha>1\),
\[
    \sum_{k=1}^{\infty}\gamma_k<\infty
    \qquad\text{almost surely}.
\]
Consequently, the partial Minkowski sums
\(\bigoplus_{k\leq n}\gamma_k[0,\eps_k]\) converge in the Hausdorff metric to
the random compact convex set
\begin{equation}\label{eq:zonotope}
    Z_\alpha=\bigoplus_{k=1}^{\infty}\gamma_k[0,\eps_k].
\end{equation}

This is one of the standard LePage series in the cone of compact convex sets;
the resulting random compact set is strictly
\(\alpha\)-stable~\nbcite{davydov2000,davydov2008}.

\begin{remark}[role of the condition \(\alpha<1\)]\label{rem:alpha-range}
The restriction \(0<\alpha<1\) is not merely technical. By the theorem of
Gin{\'e} and Hahn~\nbcite{gine1985} (see also
\nbcite[Theorem~2.15]{molchanov2005}), every \(\alpha\)-stable random compact
convex set with \(1\leq\alpha\leq2\) can be represented as
\(K\oplus\{\xi\}\), where \(K\) is a deterministic compact convex set and
\(\xi\) is an \(\alpha\)-stable random vector. Its shape is deterministic, so
no random fractal boundary geometry arises for \(\alpha\geq1\). Moreover, when
\(\alpha<1\),
\(\int_0^1 r\,\alpha r^{-1-\alpha}\dd r=\alpha/(1-\alpha)<\infty\), and the
small-jump part of the corresponding Poisson integral therefore has finite
variation and requires no compensation~\nbcite{samorodnitsky1994,bertoin1996}.
\end{remark}

\begin{definition}[extreme point]\label{def:extreme}
A point \(x\) of a compact convex set \(K\subset\R^d\) is called
\emph{extreme} if \(x=ty+(1-t)z\), with \(y,z\in K\) and \(t\in(0,1)\),
implies \(y=z=x\). The set of extreme points is denoted by \(\ext K\).
\end{definition}

\begin{definition}[exposed point]\label{def:exposed}
A point \(x\) of a compact convex set \(K\subset\R^d\) is called
\emph{exposed} if there exists \(u\in\Sph\) such that the linear functional
\(y\mapsto\langle u,y\rangle\) attains its maximum over \(K\) only at \(x\).
The set of exposed points is denoted by \(\expo K\).
\end{definition}

\begin{remark}[exposed points are extreme]\label{rem:exposed-extreme}
Every exposed point of a compact convex set is extreme. Equivalently, every
compact convex set \(K\subset\R^d\) satisfies
\[
    \expo K\subset\ext K.
\]
A proof of this standard fact is given in
Proposition~\ref{prop:exposed-implies-extreme}.
\end{remark}

\begin{remark}[the reverse inclusion fails]\label{rem:extreme-not-exposed}
In general, an extreme point need not be exposed:
\[
    \ext K\not\subset\expo K.
\]
A counterexample, its proof, and an illustration are given in
Proposition~\ref{prop:extreme-not-exposed}.
\end{remark}

Our main object is the random set \(F_\alpha\eq\ext Z_\alpha\).

\subsection{Dimension and energy}

For \(A\subset\R^d\), \(s>0\), and \(\delta>0\), set
\[
    \mathcal H^s_\delta(A)
    =
    \inf\left\{
        \sum_j(\diam U_j)^s:
        A\subset\bigcup_j U_j,\quad
        \varnothing\neq U_j\subset\R^d,\quad
        \diam U_j\leq\delta
    \right\},
\]
where the infimum is taken over all countable families of nonempty sets
\(U_j\subset\R^d\). We write
\(\mathcal H^s(A)=\lim_{\delta\downarrow0}\mathcal H^s_\delta(A)\) for the
Hausdorff measure and
\(\dimH A=\inf\{s>0:\mathcal H^s(A)=0\}\) for the Hausdorff dimension. We
shall use monotonicity and countable stability of Hausdorff
dimension~\nbcite{falconer2014,mattila1995}.

\begin{definition}[energy of a measure]\label{def:energy}
Let \(\mu\) be a finite Borel measure on \(\R^d\) and let \(s>0\). Its
\emph{\(s\)-energy} is
\[
    I_s(\mu)
    \eq
    \iint\|x-y\|^{-s}\,\mu(\dd x)\,\mu(\dd y)\in[0,\infty].
\]
\end{definition}

Energy and dimension are related by Frostman's energy criterion: if a finite
nonzero Borel measure \(\mu\) is supported on a set \(A\) and
\(I_s(\mu)<\infty\), then \(\dimH A\geq s\). The precise statement and a
reference are given in Proposition~\ref{prop:frostman}.

\begin{theorem}[Main result]\label{thm:main}
Let \(d\geq2\), \(0<\alpha<1\), and suppose that the directions \(\eps_k\)
are uniformly distributed on \(\Sph\). Then, almost surely,
\[
    \dimH F_\alpha=(d-1)\alpha.
\]
\end{theorem}

\begin{remark}[consistency with boundary geometry]\label{rem:consistency}
Since \(0<\alpha<1\), the value \((d-1)\alpha<d-1\) is consistent with the
fact that \(F_\alpha\) lies on the boundary and
\(\mathcal H^{d-1}(F_\alpha)=0\) almost surely~\nbcite{davydov2019}.
Theorem~\ref{thm:main} gives the exact dimension, while
Corollary~\ref{cor:finite-measure} establishes finiteness of the critical
Hausdorff measure.
\end{remark}

\subsection*{Idea of the proof}

Both bounds are obtained from the field \(X\) on the sphere defined
in~\eqref{eq:field}.

The support face of the zonotope in direction \(u\) (see the definition in
Subsection~\ref{sec:support-faces}) is a singleton precisely when none of the
\(\eps_k\) is orthogonal to \(u\); in that case it equals \(X(u)\). By
Lemma~\ref{lem:parametrization}, the exposed points form \(X(G)\), while the
extreme points lie between \(X(G)\) and \(\overline{X(\Sph)}\). The problem is
thus reduced to two-sided dimension estimates for the image of \(X\).

For the \emph{lower bound}, the increment \(X(u)-X(v)\) must not become too
small. This increment is a stable vector to which only directions in the
symmetric difference of two hemispheres contribute, namely a spherical lune
of width comparable to \(\theta=d_S(u,v)\). Within the lune, the tangential
component of a direction is of order one, whereas its normal component is of
order \(\theta\). The full increment is therefore anisotropic, and after the
normalization \(\theta^{-1/\alpha}\) its normal component degenerates. The
tangential projection
\(Y(u,v)=P_u(X(u)-X(v))\) defined in
\eqref{eq:tangent-projection-poisson} removes this degeneracy: its
characteristic exponent is uniformly comparable to \(\theta\|t\|^\alpha\)
in all tangent directions \(t\), so the normalized projection has a uniformly
bounded density in the \(N\)-dimensional tangent space. This yields negative
moment estimates, finiteness of the \(s\)-energy of the image of surface
measure for \(s<N\alpha\), and the lower bound by Frostman's criterion.

For the \emph{upper bound}, a direct oscillation estimate for the field is
insufficient. The total oscillation on a cell \(Q\) of scale \(h\) from
\eqref{eq:chart-cover} has a tail of order \(\alpha\), because it may be caused
by a single large jump, and summing such estimates over \(h^{-N}\) cells does
not give the desired exponent. A finite number of large jumps merely
\emph{cuts} the parameter space and does not \emph{stretch} the images of its
individual pieces. Under the central projection from Subsection~\ref{sec:charts},
great spheres become affine hyperplanes. Thus the large jumps determine a
finite hyperplane arrangement, whose number of sign strata is controlled by
Lemma~\ref{lem:arrangement}. The large-jump contribution is constant on each
stratum, so the diameter of its image is bounded by the variation of the small
jumps. It remains to match the threshold \(\delta\) between large and small
jumps to the scale \(h\): the choice \(\delta=h^{1/\alpha}\) in
\eqref{eq:delta-scale} gives
\(\E M_Q\leq C\) and \(\E V_Q^{\,s}\leq C h^{s/\alpha}\), where \(M_Q,V_Q\)
are introduced in Subsection~\ref{sec:large-small}. Summation over the
\(h^{-N}\) cells converges exactly when \(s>N\alpha\).

The same calculation at \(s=N\alpha\) gives a uniformly bounded, though not
vanishing, expectation. Fatou's lemma then yields finiteness of the critical
Hausdorff measure.

\section{Geometric parametrization}\label{sec:parametrization}

\subsection{Support faces}\label{sec:support-faces}

For a compact convex set \(K\subset\R^d\), define its support function and
support face by
\[
    h_K(u)=\max_{x\in K}\langle u,x\rangle,
    \qquad
    K(u)=\{x\in K:\langle u,x\rangle=h_K(u)\}.
\]
For two compact convex sets,
\begin{equation}\label{eq:support-additivity}
    h_{A\oplus B}(u)=h_A(u)+h_B(u),
    \qquad
    (A\oplus B)(u)=A(u)\oplus B(u);
\end{equation}
both identities follow immediately from the definition of the maximum of a
linear functional.

Define the field
\begin{equation}\label{eq:field}
    X(u)=\sum_{k=1}^{\infty}
    \gamma_k\eps_k\1_{\{\langle u,\eps_k\rangle>0\}},
    \qquad u\in\Sph.
\end{equation}
The series converges absolutely and uniformly in \(u\), since the norm of its
\(k\)th term is at most \(\gamma_k\). The map
\((\omega,u)\mapsto X(\omega,u)\) is jointly measurable as the pointwise
limit of jointly measurable partial sums.

For a fixed realization of the directions, set
\[
    G=\Sph\setminus\bigcup_{k=1}^{\infty}\eps_k^\perp,
    \qquad
    \eps_k^\perp=\{u\in\Sph:\langle u,\eps_k\rangle=0\}.
\]
The set \(G\) of \emph{good directions} has full \(\sigma\)-measure because
every equator \(\eps_k^\perp\) has measure zero.

\begin{lemma}[parametrization of the exposed points]\label{lem:parametrization}
Almost surely, the following relations hold:
\begin{equation}\label{eq:sandwich}
    X(G)=\expo Z_\alpha
    \subset F_\alpha
    \subset\overline{\expo Z_\alpha}
    \subset\overline{X(\Sph)}.
\end{equation}
\end{lemma}

\begin{proof}
Let \(u\in G\). Split the series~\eqref{eq:zonotope} into two compact sets:
write \(Z_n=\bigoplus_{k\leq n}\gamma_k[0,\eps_k]\) and
\(R_n=\bigoplus_{k>n}\gamma_k[0,\eps_k]\), so that
\(Z_\alpha=Z_n\oplus R_n\). On every segment
\(\gamma_k[0,\eps_k]\), the functional \(\langle u,\cdot\rangle\) attains
its maximum at a unique point: at \(\gamma_k\eps_k\) if
\(\langle u,\eps_k\rangle>0\), and at \(0\) otherwise. Thus
\(Z_n(u)\) is a singleton by~\eqref{eq:support-additivity}. Furthermore,
\(\diam R_n\leq\sum_{k>n}\gamma_k\), whence
\[
    \diam Z_\alpha(u)
    =\diam\bigl(Z_n(u)\oplus R_n(u)\bigr)
    \leq\diam Z_n(u)+\diam R_n(u)
    \leq\sum_{k>n}\gamma_k .
\]
The right-hand side tends to zero, so \(Z_\alpha(u)\) is a singleton. Its
unique point is the sum of the selected endpoints, namely \(X(u)\).

Now let \(u\notin G\), and suppose that
\(\langle u,\eps_k\rangle=0\) for some \(k\). Write
\(Z_\alpha=\gamma_k[0,\eps_k]\oplus Z^{(k)}\), where
\(Z^{(k)}=\bigoplus_{j\neq k}\gamma_j[0,\eps_j]\). By
\eqref{eq:support-additivity},
\(Z_\alpha(u)=\gamma_k[0,\eps_k]\oplus Z^{(k)}(u)\). This set contains a
translate of a nondegenerate segment and is therefore not a singleton. Hence
the exposing directions are precisely the elements of \(G\), proving the
first equality.

By Remark~\ref{rem:exposed-extreme}, every exposed point is extreme, whereas
Straszewicz's theorem (Proposition~\ref{prop:straszewicz}) states that the
exposed points are dense in the set of extreme points of a compact convex
set. Together with the equality just proved, this gives~\eqref{eq:sandwich}.
\end{proof}

\begin{remark}[passing to the closure]\label{rem:closure}
The rightmost inclusion in~\eqref{eq:sandwich} is the only place where a
closure appears, and it requires care: in general, taking the closure may
increase Hausdorff dimension, as shown by
\(\dimH\mathbb Q=0\) and \(\dimH\overline{\mathbb Q}=1\). Here this obstacle
is avoided at no cost. The coverings constructed in Section~\ref{sec:upper}
are finite, the closure of a finite union is the union of the closures, and
diameters are unchanged under closure. Thus the upper bound is proved directly
for \(\overline{X(\Sph)}\). The lower bound, in turn, will already be proved
for \(X(G)\).
\end{remark}

\subsection{Poisson representation}

The points \(\Gamma_k\) form a standard Poisson process on \((0,\infty)\)
with intensity \(\dd\Gamma\). The substitution
\(r=\Gamma^{-1/\alpha}\), or \(\Gamma=r^{-\alpha}\), transforms this measure
into \(\alpha r^{-1-\alpha}\dd r\). Independently marking the points by
directions with law \(\sigma\) shows that
\[
    \Pi=\sum_{k=1}^{\infty}\delta_{(\gamma_k,\eps_k)}
\]
is a Poisson random measure on \((0,\infty)\times\Sph\) with intensity
\begin{equation}\label{eq:intensity}
    \nu(\dd r,\dd x)
    =
    \alpha r^{-1-\alpha}\dd r\,\sigma(\dd x),
\end{equation}
and the series~\eqref{eq:field} has the integral form
\begin{equation}\label{eq:field-poisson}
    X(u)=
    \int_{(0,\infty)\times\Sph}
    r x\1_{\{\langle u,x\rangle>0\}}\,\Pi(\dd r,\dd x).
\end{equation}
\section{Lower bound}\label{sec:lower}

\subsection{Tangential projection of an increment}\label{sec:tangent-projection}

For \(u,v\in\Sph\), set
\[
    \Delta(u,v)\eq X(u)-X(v),
    \qquad
    a_{u,v}(x)
    \eq
    \1_{\{\langle u,x\rangle>0\}}
    -
    \1_{\{\langle v,x\rangle>0\}}.
\]
Then~\eqref{eq:field-poisson} gives
\[
    \Delta(u,v)
    =
    \int r x\,a_{u,v}(x)\,\Pi(\dd r,\dd x).
\]
This is a strictly \(\alpha\)-stable random vector. Let \(P_u\) denote the
orthogonal projection onto \(u^\perp\). The tangential projection of the
increment has the integral representation
\begin{equation}\label{eq:tangent-projection-poisson}
    Y(u,v)
    \eq
    P_u\Delta(u,v)
    =
    \int r\,P_ux\,a_{u,v}(x)\,\Pi(\dd r,\dd x).
\end{equation}
Since \(\|Y(u,v)\|\leq\|\Delta(u,v)\|\), it follows that, for every
\(\beta>0\),
\begin{equation}\label{eq:projection-negative}
    \|\Delta(u,v)\|^{-\beta}
    \leq
    \|Y(u,v)\|^{-\beta}.
\end{equation}

The geometric meaning of this projection and the reason for replacing the
full increment \(\Delta(u,v)\) by \(Y(u,v)\) are explained in greater detail
in Remark~\ref{rem:why-projection} of Appendix~\ref{app:technical}.

\begin{lemma}[characteristic exponent]\label{lem:characteristic}
For \(t\in u^\perp\),
\begin{equation}\label{eq:char-modulus}
    \bigl|\E e^{i\langle t,Y(u,v)\rangle}\bigr|
    =
    \exp\{-c_\alpha\Psi_{u,v}(t)\},
    \qquad
    \Psi_{u,v}(t)
    \eq
    \int_{\Sph}
    |\langle t,x\rangle|^\alpha\,|a_{u,v}(x)|\,\sigma(\dd x),
\end{equation}
where
\(c_\alpha=\alpha\int_0^\infty(1-\cos s)s^{-1-\alpha}\dd s>0\).
\end{lemma}

\begin{proof}
Since \(t\in u^\perp\), we have
\(\langle t,P_ux\rangle=\langle t,x\rangle\). Hence
\eqref{eq:tangent-projection-poisson} implies
\[
    \langle t,Y(u,v)\rangle
    =
    \int r\langle t,x\rangle a_{u,v}(x)\,\Pi(\dd r,\dd x).
\]
The exponential formula for a Poisson random measure gives
\[
    \log\E e^{i\langle t,Y(u,v)\rangle}
    =
    \int_0^\infty\!\!\int_{\Sph}
    \left(
        e^{ir\langle t,x\rangle a_{u,v}(x)}-1
    \right)
    \alpha r^{-1-\alpha}\sigma(\dd x)\dd r.
\]
Taking real parts and using the identity
\[
    \alpha\int_0^\infty
    \bigl(1-\cos(ra)\bigr)r^{-1-\alpha}\dd r
    =
    c_\alpha|a|^\alpha
\]
with \(a=\langle t,x\rangle a_{u,v}(x)\), we obtain
\eqref{eq:char-modulus}. Here we used
\(a_{u,v}\in\{-1,0,1\}\), and hence
\(|a_{u,v}|^\alpha=|a_{u,v}|\).
\end{proof}

\begin{remark}[asymmetry of the increment]\label{rem:asymmetry}
In general, the law of \(\Delta(u,v)\) is not symmetric. Indeed,
\(a_{u,v}(-x)=-a_{u,v}(x)\), so
\((-x)a_{u,v}(-x)=x\,a_{u,v}(x)\): antipodal directions make equal rather
than opposite contributions. For this reason, the representation
\(\E e^{i\langle t,S\rangle}=e^{-\Psi(t)}\) with real-valued \(\Psi\) does
not apply here, and Lemma~\ref{lem:characteristic} is deliberately stated for
the modulus of the characteristic function. This is sufficient for the
density and negative-moment estimates.
\end{remark}

\subsection{Uniform tangential non-degeneracy}

Let \(\theta=d_S(u,v)\). By rotational invariance of \(\sigma\), it suffices
to consider the canonical pair
\[
    u=e_d,
    \qquad
    v_\theta=\cos\theta\,e_d+\sin\theta\,e_{d-1}.
\]
Indeed, if \(R\in O(d)\) maps \((u,v)\) to \((e_d,v_\theta)\), then
\(\{\eps_k\}\overset{d}{=}\{R^{-1}\eps_k\}\), and consequently
\(Y(u,v)\overset{d}{=}R^{-1}Y(e_d,v_\theta)\); in particular, the norms have
the same distribution. Identify \(e_d^\perp\) with \(\R^N\) and write
\(\Psi_\theta\eq\Psi_{e_d,v_\theta}\).

\begin{lemma}[tangential non-degeneracy]\label{lem:tangent}
There exist constants \(0<c\leq C<\infty\), depending only on \(d\) and
\(\alpha\), such that, for all \(0<\theta\leq\pi/4\) and all
\(t\in e_d^\perp\),
\begin{equation}\label{eq:psi-comparison}
    c\,\theta\|t\|^\alpha
    \leq
    \Psi_\theta(t)
    \leq
    C\,\theta\|t\|^\alpha.
\end{equation}
\end{lemma}

The proof is a direct computation in spherical coordinates and is given in
Appendix~\ref{app:psi}; the case \(d=2\) is discussed there in
Remark~\ref{rem:d2}.

\subsection{Negative moments}

\begin{lemma}[uniform density bound]\label{lem:bounded-densities}
Let \(0<\theta\leq\pi/4\) and
\(W_\theta\eq\theta^{-1/\alpha}Y(e_d,v_\theta)\). Then \(W_\theta\) has a
continuous density \(p_\theta\) on \(\R^N\), and
\[
    \sup_{0<\theta\leq\pi/4}\ \sup_{x\in\R^N}p_\theta(x)<\infty.
\]
\end{lemma}

\begin{proof}
The function \(\Psi_\theta\) is homogeneous of degree \(\alpha\), so
\(\Psi_\theta(\theta^{-1/\alpha}t)=\theta^{-1}\Psi_\theta(t)\). Together
with Lemmas~\ref{lem:characteristic} and~\ref{lem:tangent}, this yields
\[
    \bigl|\E e^{i\langle t,W_\theta\rangle}\bigr|
    =
    \exp\bigl\{-c_\alpha\theta^{-1}\Psi_\theta(t)\bigr\}
    \leq
    \exp\bigl\{-c_\alpha c\,\|t\|^\alpha\bigr\},
\]
with a constant independent of \(\theta\). The right-hand side is integrable
over \(\R^N\), so Fourier inversion gives a continuous density and the bound
\[
    \|p_\theta\|_\infty
    \leq
    (2\pi)^{-N}\int_{\R^N}e^{-c_\alpha c\|t\|^\alpha}\dd t<\infty,
\]
whose right-hand side is independent of \(\theta\).
\end{proof}

\begin{proposition}[negative moments]\label{prop:negative-moments}
Let \(U=B_S(u_0,r)\subset\Sph\) be a geodesic ball, that is, the ball of
radius \(r\) centered at \(u_0\) in the metric \(d_S\), and suppose that
\(\diam U\leq\pi/4\). For every \(0<\beta<N\), there exists
\(C_\beta<\infty\), depending only on \(d\), \(\alpha\), and \(\beta\),
such that, for all distinct \(u,v\in U\),
\begin{equation}\label{eq:negative-moment}
    \E\|X(u)-X(v)\|^{-\beta}
    \leq
    C_\beta\, d_S(u,v)^{-\beta/\alpha}.
\end{equation}
\end{proposition}

\begin{proof}
First take \((u,v)=(e_d,v_\theta)\). By
Lemma~\ref{lem:bounded-densities},
\[
    \E\|W_\theta\|^{-\beta}
    =
    \int_{\R^N}\|x\|^{-\beta}p_\theta(x)\dd x
    \leq
    \|p_\theta\|_\infty\int_{\|x\|\leq1}\|x\|^{-\beta}\dd x+1
    \leq C_\beta,
\]
because the integral over the unit ball is finite exactly when \(\beta<N\),
while \(\|x\|^{-\beta}\leq1\) on its complement. Since
\(Y(e_d,v_\theta)=\theta^{1/\alpha}W_\theta\), it follows that
\(\E\|Y(e_d,v_\theta)\|^{-\beta}\leq C_\beta\theta^{-\beta/\alpha}\).
Rotational invariance extends this estimate to arbitrary \(u,v\in U\) with
\(\theta=d_S(u,v)\leq\pi/4\), and~\eqref{eq:projection-negative} then
gives~\eqref{eq:negative-moment}.
\end{proof}

\subsection{Energy argument}

\begin{theorem}[lower bound]\label{thm:lower}
Almost surely,
\begingroup
\setlength{\abovedisplayskip}{3pt}
\setlength{\belowdisplayskip}{2pt}
\setlength{\abovedisplayshortskip}{3pt}
\setlength{\belowdisplayshortskip}{2pt}
\[
    \dimH F_\alpha\geq N\alpha=(d-1)\alpha.
\]
\endgroup
\end{theorem}

\begin{proof}
Fix a geodesic ball \(U\subset\Sph\) of diameter at most \(\pi/4\), and let
\(\lambda\) be the restriction of \(\sigma\) to \(U\); this is a finite
nonzero Borel measure. The set \(G\cap U\) is Borel and has full
\(\lambda\)-measure, and the map \(X\) is Borel. Therefore \(X(G\cap U)\) is
analytic and, in particular, universally measurable. These terms and this
implication are explained in Remark~\ref{rem:analytic-universal} of
Appendix~\ref{app:standard}. Consider the random finite measure
\(\mu=X_{\#}\lambda\), the pushforward of \(\lambda\) under \(X\). Since
\(\lambda(U\setminus G)=0\), the measure \(\mu\) is supported on
\(X(G\cap U)\subset F_\alpha\), and
\(\mu(\R^d)=\lambda(U)>0\).

Fix \(0<s<N\alpha\). Then \(s<N\), and a change of variables, Tonelli's
theorem, and Proposition~\ref{prop:negative-moments} give
\begin{align*}
    \E I_s(\mu)
    &=
    \int_{U\times U}
    \E\|X(u)-X(v)\|^{-s}\,\lambda(\dd u)\,\lambda(\dd v)\\
    &\leq
    C_s\int_{U\times U}
    d_S(u,v)^{-s/\alpha}\,\lambda(\dd u)\,\lambda(\dd v).
\end{align*}
The last integral is finite: the ball \(U\) is bi-Lipschitz equivalent to a
ball in \(\R^N\), and
\(\int_0^1 r^{N-1-s/\alpha}\dd r\) converges exactly when
\(s/\alpha<N\). Hence \(I_s(\mu)<\infty\) almost surely.

Proposition~\ref{prop:frostman} now implies almost surely that
\(\dimH X(G\cap U)\geq s\), and therefore \(\dimH F_\alpha\geq s\). Taking
the intersection of the probability-one events over all rational
\(s<N\alpha\) proves the claim.
\end{proof}

\section{Upper bound}\label{sec:upper}

\subsection{Local charts and spherical belts}\label{sec:charts}

For \(x\in\Sph\), let
\(H(x)=\{u\in\Sph:\langle u,x\rangle=0\}\) be the great sphere across which
the corresponding indicator in~\eqref{eq:field} changes value.

We use coordinate charts obtained by centrally projecting a hemisphere onto
its tangent hyperplane. After an orthogonal change of coordinates, on the
hemisphere \(u_d>0\) this projection is
\[
    u=(u',u_d)\longmapsto t=\frac{u'}{u_d}\in\R^N,
\]
and the inverse chart is
\[
    \phi(t)=\frac{(t,1)}{\sqrt{1+\|t\|^2}},
    \qquad t\in\mathcal D\subset\R^N,
\]
where \(\mathcal D\) is a compact cube. The condition
\(\langle\phi(t),x\rangle>0\) is equivalent to
\(\langle(t,1),x\rangle>0\), that is, to positivity of an affine function of
\(t\). The map \(\phi\) is bi-Lipschitz on bounded cubes, and the sphere can
be covered by finitely many charts of this form.

Partitioning the chart domains into subcubes of side length \(h=2^{-n}\) and
taking their images gives a finite covering \(\mathcal Q_n\) of the sphere by
connected compact sets \(Q\) such that
\begin{equation}\label{eq:chart-cover}
    \#\mathcal Q_n\leq C h^{-N},
    \qquad
    \diam Q\leq C h.
\end{equation}

For \(Q\in\mathcal Q_n\), define the set of directions whose great spheres
meet \(Q\), and let \(p_Q\) be its measure:
\[
    A(Q)=\{x\in\Sph:H(x)\cap Q\neq\varnothing\},
    \qquad
    p_Q=\sigma(A(Q)).
\]
Since \(Q\) is compact,
\(A(Q)=\{x:\min_{v\in Q}\langle v,x\rangle\leq0\leq
\max_{v\in Q}\langle v,x\rangle\}\) is closed and hence measurable.

\begin{lemma}[belt estimate]\label{lem:belt}
Uniformly in \(n\) and \(Q\in\mathcal Q_n\),
\begin{equation}\label{eq:belt}
    p_Q\leq C h.
\end{equation}
\end{lemma}

\begin{proof}
Choose \(u_0\in Q\). If \(x\in A(Q)\), then
\(\langle v,x\rangle=0\) for some \(v\in Q\), and therefore
\[
    |\langle u_0,x\rangle|
    =
    |\langle u_0-v,x\rangle|
    \leq\|u_0-v\|
    \leq Ch .
\]
Thus \(A(Q)\) is contained in a spherical belt of width \(Ch\) around the
equator \(u_0^\perp\), whose measure is at most \(C'h\)
(Proposition~\ref{prop:belt-density}).
\end{proof}

\subsection{Hyperplane arrangements}

\begin{lemma}[number of sign strata]\label{lem:arrangement}
Let \(\mathcal D\subset\R^N\) be convex and let
\(\ell_1,\ldots,\ell_m\) be affine functions on \(\R^N\). The partition of
\(\mathcal D\) into nonempty level sets of the vector-valued map
\[
    t\longmapsto\bigl(\1_{\{\ell_1(t)>0\}},\ldots,
    \1_{\{\ell_m(t)>0\}}\bigr)
\]
has at most \(C_N(m+1)^N\) members. We call these sets \emph{sign strata};
the count includes strata contained in the hyperplanes
\(\{\ell_j=0\}\) themselves.
\end{lemma}

The proof is given in Appendix~\ref{app:arrangement}.

\subsection{Large and small jumps}\label{sec:large-small}

Fix the scale \(h=2^{-n}\) and set
\begin{equation}\label{eq:delta-scale}
    \delta\eq h^{1/\alpha}.
\end{equation}
For \(Q\in\mathcal Q_n\), let \(M_Q\) be the number of \emph{large} jumps
whose great spheres meet \(Q\), and let \(V_Q\) be the total variation of the
corresponding \emph{small} jumps:
\[
    M_Q
    \eq
    \Pi\bigl((\delta,\infty)\times A(Q)\bigr),
    \qquad
    V_Q
    \eq
    \int_{(0,\delta]\times A(Q)}
    r\,\Pi(\dd r,\dd x).
\]

\begin{lemma}[small-jump scale]\label{lem:small-jumps}
For every \(q>0\), there exists \(C_q<\infty\), depending only on \(d\),
\(\alpha\), and \(q\), such that
\begin{equation}\label{eq:small-moment}
    \E\bigl[(1+M_Q)^N V_Q^{\,q}\bigr]
    \leq C_q\,\delta^{q}
\end{equation}
uniformly over all positive integers \(n\) and all \(Q\in\mathcal Q_n\).
\end{lemma}

\begin{proof}
By~\eqref{eq:intensity}, the random variable \(M_Q\) has a Poisson
distribution with parameter
\[
    p_Q\int_\delta^\infty\alpha r^{-1-\alpha}\dd r
    =
    p_Q\,\delta^{-\alpha}
    \leq
    C h\cdot h^{-1}
    =
    C.
\]
Here we used~\eqref{eq:belt} and~\eqref{eq:delta-scale}; the threshold
\(\delta=h^{1/\alpha}\) was chosen precisely so that \(C\) is independent of
\(n\). Consequently, all moments of \((1+M_Q)^N\) are bounded uniformly in
\(n\) and \(Q\).

The random variables \(M_Q\) and \(V_Q\) are independent because they
correspond to the disjoint regions \((\delta,\infty)\times A(Q)\) and
\((0,\delta]\times A(Q)\). For \(t>0\), the exponential formula for a Poisson
integral and the substitution \(r=\delta y\) give
\begin{align*}
    \E\exp\left\{t\frac{V_Q}{\delta}\right\}
    &=
    \exp\left\{
        p_Q\int_0^\delta
        \left(e^{tr/\delta}-1\right)
        \alpha r^{-1-\alpha}\dd r
    \right\}\\
    &=
    \exp\left\{
        \alpha\,p_Q\delta^{-\alpha}
        \int_0^1(e^{ty}-1)y^{-1-\alpha}\dd y
    \right\}.
\end{align*}
The integral is finite because \(e^{ty}-1\asymp ty\) as \(y\to0\) and
\(\alpha<1\), while the factor \(p_Q\delta^{-\alpha}\) is uniformly bounded.
Thus all positive moments of \(V_Q/\delta\) are uniformly bounded, and
independence yields~\eqref{eq:small-moment}.
\end{proof}

\subsection{Adaptive covering of the image}

\begin{lemma}[adaptive covering of a cell]\label{lem:adaptive-cover}
For every \(Q\in\mathcal Q_n\), the set \(X(Q)\) can be covered by at most
\(C_N(1+M_Q)^N\) sets of diameter at most \(V_Q\).
\end{lemma}

This covering is illustrated in Fig.~\ref{fig:cell}.

\begin{proof}
The number of large jumps is finite almost surely. Indeed,
\[
    \Pi\bigl((\delta,\infty)\times\Sph\bigr)
\]
has a Poisson distribution with mean \(\delta^{-\alpha}\).

If an atom \((r,x)\), with \(r>\delta\), has \(x\notin A(Q)\), then the
function \(\langle\cdot,x\rangle\) does not vanish on the connected set \(Q\)
and therefore has constant sign; the contribution of this atom to \(X(u)\)
is constant on \(Q\). The remaining large atoms are the \(M_Q\) atoms with
\(x\in A(Q)\). Let \(\mathcal D_Q\) be the chart subcube such that
\(\phi(\mathcal D_Q)=Q\). In this chart, the condition
\(\langle\phi(t),x\rangle>0\) is equivalent to \(\ell_x(t)>0\), where
\(\ell_x(t)=\langle(t,1),x\rangle\) is affine. By
Lemma~\ref{lem:arrangement}, these \(M_Q\) functions divide
\(\mathcal D_Q\) into at most \(C_N(1+M_Q)^N\) sign strata. The contribution
of all large jumps is constant on each stratum, including strata contained in
the hyperplanes themselves.

Let \(u,v\in Q\) be the images of two points in the same stratum. Only the
indicators of small jumps can differ, and only for jumps whose great spheres
meet \(Q\). Therefore
\[
    \|X(u)-X(v)\|
    \leq
    \int_{(0,\delta]\times A(Q)}
    r\,\Pi(\dd r,\dd x)
    =
    V_Q .
\]
The images of the sign strata form the required covering.
\end{proof}

\begin{figure}[ht]
\centering
\begin{tikzpicture}[x=1cm,y=1cm,>=Stealth]
\fill[blue!4] (0,0) rectangle (4,4);
\fill[red!12] (0,0.93)--(1.78,1.53)--(2.28,2.62)--(0,3.28)--cycle;
\draw[thick] (0,0) rectangle (4,4);
\draw[blue!65!black,thick] (-0.4,0.8)--(4.4,2.4);
\draw[blue!65!black,thick] (0.9,-0.4)--(3.1,4.4);
\draw[blue!65!black,thick] (-0.4,3.4)--(4.4,2.0);
\node[blue!65!black,font=\small,right] at (4.42,2.4) {\(\ell_1=0\)};
\node[blue!65!black,font=\small,above] at (3.1,4.4) {\(\ell_2=0\)};
\node[blue!65!black,font=\small,left] at (-0.42,3.4) {\(\ell_3=0\)};
\node[font=\small] at (0.78,2.24) {stratum};
\draw[<->,thick] (0,-0.3)--(4,-0.3);
\node[font=\small,below] at (2,-0.3) {\(h\)};
\node[font=\small] at (3.5,0.5) {\(\mathcal D_Q\)};
\end{tikzpicture}
\caption{A cell in central-projection coordinates. The traces of the great
spheres associated with large jumps are affine hyperplanes \(\ell_j=0\),
which divide the subcube into at most \(C_N(1+M_Q)^N\) sign strata. On each
stratum (one is shaded), the large-jump contribution is constant, so the
diameter of its image is at most \(V_Q\).}\label{fig:cell}
\end{figure}
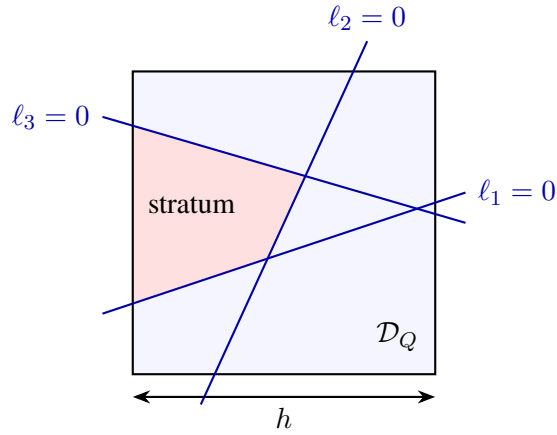

\begin{theorem}[upper bound]\label{thm:upper}
Almost surely,
\[
    \dimH\overline{X(\Sph)}\leq N\alpha .
\]
In particular, \(\dimH F_\alpha\leq(d-1)\alpha\).
\end{theorem}

\begin{proof}
\step{Step 1: constructing the covering.}
Fix \(s>N\alpha\). For every \(n\), apply
Lemma~\ref{lem:adaptive-cover} to all \(Q\in\mathcal Q_n\). Since
\(\bigcup_{Q\in\mathcal Q_n}Q=\Sph\), the union of the resulting families is
a finite covering of \(X(\Sph)\). Replacing its members by their closures gives
a covering of \(\overline{X(\Sph)}\): the closure of a finite union is the
union of the closures, and closure does not change diameter. Set
\[
    T_n
    \eq
    C_N\sum_{Q\in\mathcal Q_n}(1+M_Q)^N V_Q^{\,s},
    \qquad
    \eta_n\eq\max_{Q\in\mathcal Q_n}V_Q .
\]
Then the sum of the \(s\)th powers of the diameters in the \(n\)th covering is
at most \(T_n\), and each member has diameter at most \(\eta_n\).

\step{Step 2: estimating the expectation.}
By~\eqref{eq:chart-cover}, Lemma~\ref{lem:small-jumps}, and
\eqref{eq:delta-scale},
\begin{equation}\label{eq:ET}
    \E T_n
    \leq
    C h^{-N}\delta^{\,s}
    =
    C h^{\,s/\alpha-N}.
\end{equation}
For \(h=2^{-n}\) and \(s/\alpha>N\), the series \(\sum_n\E T_n\) is
geometrically convergent.

\step{Step 3: passing to an almost-sure statement.}
By Tonelli's theorem,
\(\E\sum_n T_n=\sum_n\E T_n<\infty\). Thus \(\sum_n T_n<\infty\) almost
surely and, in particular, \(T_n\to0\). Assuming without loss of generality
that \(C_N\geq1\), we have \(\eta_n^{\,s}\leq T_n\), and hence
\(\eta_n\to0\).

\step{Step 4: conclusion.}
We have
\(\mathcal H^s_{\eta_n}\bigl(\overline{X(\Sph)}\bigr)\leq T_n\), where
\(\eta_n\to0\) and \(T_n\to0\) almost surely. Therefore
\(\mathcal H^s\bigl(\overline{X(\Sph)}\bigr)=0\) almost surely. Taking the
intersection of the probability-one events over all rational
\(s>N\alpha\) proves the first bound; the second follows from
\(F_\alpha\subset\overline{X(\Sph)}\) in~\eqref{eq:sandwich}.
\end{proof}

\section{Completion of the proof}\label{sec:completion}

\begin{proof}[Proof of Theorem~\ref{thm:main}]
Theorem~\ref{thm:lower} gives
\(\dimH F_\alpha\geq(d-1)\alpha\) almost surely, while
Theorem~\ref{thm:upper} gives the reverse bound. Therefore
\(\dimH F_\alpha=(d-1)\alpha\) almost surely.
\end{proof}

\begin{corollary}[all intermediate sets]\label{cor:all-images}
Almost surely,
\[
    \dimH\expo Z_\alpha
    =
    \dimH\ext Z_\alpha
    =
    \dimH X(\Sph)
    =
    \dimH\overline{X(\Sph)}
    =
    (d-1)\alpha .
\]
\end{corollary}

\begin{proof}
By Lemma~\ref{lem:parametrization}, all four sets are sandwiched between
\(X(G\cap U)\) and \(\overline{X(\Sph)}\):
\[
    X(G\cap U)\subset X(G)=\expo Z_\alpha\subset\ext Z_\alpha
    \subset\overline{X(\Sph)},
    \qquad
    X(G)\subset X(\Sph)\subset\overline{X(\Sph)}.
\]
The proof of Theorem~\ref{thm:lower} already establishes the lower bound for
\(X(G\cap U)\), whereas Theorem~\ref{thm:upper} establishes the upper bound
for \(\overline{X(\Sph)}\). Monotonicity of dimension shows that all
intermediate values coincide.
\end{proof}

The calculation used in the proof of Theorem~\ref{thm:upper} yields more than
the value of the dimension: it also controls the critical Hausdorff measure.

\begin{corollary}[finiteness of the critical measure]\label{cor:finite-measure}
Almost surely,
\[
    \mathcal H^{(d-1)\alpha}\bigl(\overline{X(\Sph)}\bigr)<\infty ;
\]
in particular,
\(\mathcal H^{(d-1)\alpha}(F_\alpha)<\infty\) almost surely.
\end{corollary}

\begin{proof}
Retain the notation from the proof of Theorem~\ref{thm:upper}, and consider
the quantities \(T_n\) corresponding to the exponent \(s=N\alpha\). Then
\(s/\alpha=N\), and~\eqref{eq:ET} gives
\(\E T_n\leq C h^0=C\) uniformly in \(n\). The series
\(\sum_n\E T_n\) need no longer converge, but Fatou's lemma gives
\[
    \E\bigl[\liminf_{n\to\infty}T_n\bigr]
    \leq
    \liminf_{n\to\infty}\E T_n
    \leq C<\infty,
\]
and hence \(\liminf_n T_n<\infty\) almost surely.

In addition, applying Step~3 of the proof of Theorem~\ref{thm:upper} with any
fixed \(s'>N\alpha\) gives \(\eta_n\to0\) almost surely.

Work on the intersection of these two probability-one events. Choose a
subsequence \(n_j\) such that \(T_{n_j}\to\liminf_n T_n\). Since
\(\eta_{n_j}\to0\), for every \(\eta>0\) there exists \(j\) with
\(\eta_{n_j}\leq\eta\); monotonicity of \(\mathcal H^s_\eta\) in \(\eta\)
then gives
\[
    \mathcal H^{N\alpha}_{\eta}\bigl(\overline{X(\Sph)}\bigr)
    \leq
    \mathcal H^{N\alpha}_{\eta_{n_j}}\bigl(\overline{X(\Sph)}\bigr)
    \leq
    T_{n_j}.
\]
Letting \(j\to\infty\), we obtain
\(\mathcal H^{N\alpha}_{\eta}(\overline{X(\Sph)})\leq\liminf_n T_n\) for
every \(\eta>0\). Finally, letting \(\eta\downarrow0\) yields
\(\mathcal H^{N\alpha}(\overline{X(\Sph)})\leq\liminf_n T_n<\infty\).
The remaining assertion follows from
\(F_\alpha\subset\overline{X(\Sph)}\).
\end{proof}

\clearpage
\section{Discussion and open problems}\label{sec:discussion}

Uniformity of the spectral measure is used differently in the two halves of
the proof. For the upper bound, it is enough to have the uniform belt estimate
\begin{equation}\label{eq:belt-general}
    \sigma\{x\in\Sph:|\langle u,x\rangle|\leq r\}\leq Cr,
    \qquad u\in\Sph,\ 0<r<1 .
\end{equation}
The central projections, the combinatorics of hyperplane arrangements, and
the scale decomposition of the jumps do not depend on \(\sigma\). Therefore
Theorem~\ref{thm:upper} and Corollary~\ref{cor:finite-measure} extend to every
spectral measure for which~\eqref{eq:belt-general} holds uniformly in \(u\),
in particular, to every measure with a bounded density relative to
\(\sigma\).

The lower bound requires rotational invariance and uniform non-degeneracy of
the tangential integral \(J\) from Appendix~\ref{app:psi}, which yields the
two-sided estimate~\eqref{eq:psi-comparison}. The first open problem is thus
to identify minimal anisotropic assumptions: the upper control of belts must
be supplemented by quantitative non-degeneracy of tangential components, and
the reduction to a canonical pair, unavailable in general, may have to be
replaced by a different normalization.

The second question concerns the critical measure. By
Corollary~\ref{cor:finite-measure}, almost surely,
\[
    \mathcal H^{(d-1)\alpha}(F_\alpha)<\infty,
\]
but the corollary does not establish positivity. For \(d=2\), the critical
measure is almost surely zero. Indeed, \(F_\alpha\) consists of two centrally
symmetric parts that are bi-Lipschitz equivalent to the range of an
\(\alpha\)-stable subordinator~\nbcite[Theorem~7]{davydov2019}. For this range,
Taylor and Wendel~\nbcite[Sections~5--6]{taylor1966} found the gauge function
\begin{equation}\label{eq:gauge}
    \varphi(h)=h^{\alpha}\bigl(\log\log(1/h)\bigr)^{1-\alpha},
\end{equation}
for which the \(\varphi\)-measure (the generalized Hausdorff measure with gauge
function \(\varphi\)) is almost surely positive and finite; in its definition,
terms of the form \(\varphi(\diam U_j)\) are summed instead of
\((\diam U_j)^s\). Since \(\varphi(h)/h^\alpha\to\infty\) as
\(h\downarrow0\), it follows that \(\mathcal H^\alpha(F_\alpha)=0\). This
suggests the conjecture
\(\mathcal H^{(d-1)\alpha}(F_\alpha)=0\) for all \(d\geq2\). For \(d\geq3\),
neither this conjecture nor the correct gauge function is known: the
one-dimensional argument uses subordinator theory and does not extend to the
\(N\)-dimensional parameter space.

Finally, the packing dimension of \(F_\alpha\) and the multifractal
properties of the field remain open. The adaptive covering contains sets of
different diameters and therefore does not directly yield an upper bound for
the Minkowski dimension; such a bound would require uniform, cell-by-cell
control of the variables \(V_Q\).

\clearpage
\appendix

\section{Technical computations}\label{app:technical}

\subsection{Geometric meaning of the tangential projection}

\begin{remark}[why the projection is needed]\label{rem:why-projection}
Return to the tangential projection introduced in
Subsection~\ref{sec:tangent-projection}:
\[
    Y(u,v)=P_u\Delta(u,v).
\]
Its integral form is given in~\eqref{eq:tangent-projection-poisson}. We now
explain why this projection is used in place of the full increment
\(\Delta(u,v)\). Only those \(\eps_k\) that belong to the symmetric difference
of the hemispheres determined by \(u\) and \(v\) contribute nontrivially to
\(\Delta(u,v)\). They form a spherical lune of width comparable to
\(\theta=d_S(u,v)\) around the equator \(u^\perp\)
(Fig.~\ref{fig:lune}). Within this lune, the tangential component of a
direction is of order one, while its normal component is of order \(\theta\).
Consequently, the characteristic exponent of the full increment is of order
\(\theta\|t\|^\alpha\) in tangential directions but only of order
\(\theta^{1+\alpha}\|t\|^\alpha\) in the normal direction. After the
normalization \(\theta^{-1/\alpha}\), which is appropriate for the tangential
part, the normal part contributes
\(\theta^\alpha\|t\|^\alpha\to0\): the normalized distribution degenerates
along \(u\), and no uniform density estimate is possible. The tangential
projection \(Y(u,v)\) discards this degenerating component and retains exactly
the part of the increment with the correct scale \(\theta^{1/\alpha}\) in an
\(N\)-dimensional space.
\end{remark}

\begin{figure}[ht]
\centering
\begin{tikzpicture}[x=1cm,y=1cm,>=Stealth]
\def\th{32}
\fill[blue!12] (0,0) -- (0:2.3) arc (0:\th:2.3) -- cycle;
\fill[blue!12] (0,0) -- (180:2.3) arc (180:180+\th:2.3) -- cycle;
\draw[thick] (0,0) circle (2.3);
\draw[blue!65!black,thick] (180:2.3)--(0:2.3);
\draw[blue!65!black,thick] (180+\th:2.3)--(\th:2.3);
\node[blue!65!black,font=\small,left] at (-2.36,0.26) {\(u^{\perp}\)};
\node[blue!65!black,font=\small,above right=1pt and 1pt] at (\th:2.3) {\(v^{\perp}\)};
\draw[->,red!70!black,thick] (0,0)--(90:2.3);
\node[red!70!black,font=\small,above] at (0,2.3) {\(u\)};
\draw[->,red!70!black,thick] (0,0)--(90+\th:2.3);
\node[red!70!black,font=\small,above left=-2pt and -2pt] at (90+\th:2.3) {\(v\)};
\draw[thin] (90:0.75) arc (90:90+\th:0.75);
\node[font=\small] at (-0.36,1.02) {\(\theta\)};
\coordinate (x) at (16:2.3);
\draw[->,thick] (0,0)--(x);
\node[font=\small,right] at (2.26,0.82) {\(x\)};
\draw[densely dashed] (x) -- (2.211,0);
\node[font=\small,right] at (2.36,0.30) {\(\langle u,x\rangle\asymp\theta\)};
\node[font=\small] at (1.05,-0.54) {\(\|P_ux\|\asymp1\)};
\node[font=\small,left] at (-2.42,-0.94) {\(\{a_{u,v}\neq0\}\)};
\draw[->,thin] (-2.34,-0.90)--(-1.62,-0.58);
\end{tikzpicture}
\caption{Section of the sphere by the plane containing \(u\) and \(v\). The
set \(\{a_{u,v}\neq0\}\) is a spherical lune of width comparable to
\(\theta\). Within it, the tangential component of a direction is of order
one, while the normal component is of order \(\theta\).}\label{fig:lune}
\end{figure}
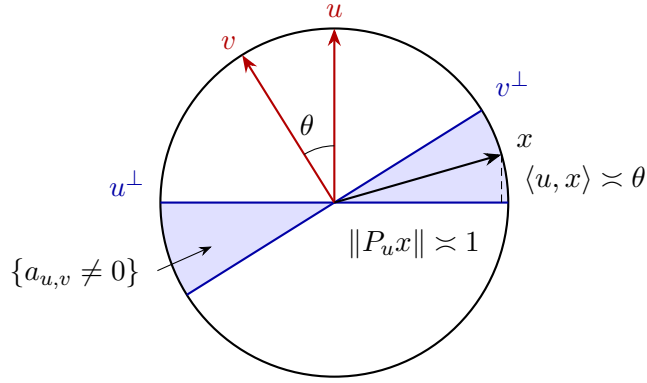

\subsection{Characteristic exponent in spherical coordinates}
\label{app:psi}

Proof of Lemma~\ref{lem:tangent}. Write \(x\in\Sph\) as
\[
    x=(\cos\varphi\,\omega,\ \sin\varphi),
    \qquad
    \omega\in S^{N-1},\quad
    -\frac{\pi}{2}\leq\varphi\leq\frac{\pi}{2},
\]
so that \(\sin\varphi\) is the last coordinate.

\step{Step 1: the interval of sign change.}
We have \(\langle e_d,x\rangle=\sin\varphi\) and
\[
    \langle v_\theta,x\rangle
    =
    \cos\theta\sin\varphi+\sin\theta\cos\varphi\,\omega_N
    =
    \sqrt{\cos^2\theta+\sin^2\theta\,\omega_N^2}\;
    \sin(\varphi+\psi),
\]
where \(\tan\psi=\tan\theta\,\omega_N\) and \(|\psi|<\pi/2\). The first
expression changes sign at \(\varphi=0\), and the second at
\(\varphi=-\psi\). Thus, for fixed \(\omega\), the signs differ precisely on
the interval between \(-\psi\) and \(0\), whose length is
\[
    b_\theta(\omega)=\arctan\bigl(\tan\theta\,|\omega_N|\bigr).
\]

\step{Step 2: integral representation.}
In these coordinates, the uniform measure on the sphere is
\[
    \sigma(\dd x)=c_d(\cos\varphi)^{N-1}\dd\varphi\,\dd\omega.
\]
Moreover, for \(t\in e_d^\perp\),
\(|\langle t,x\rangle|^\alpha=(\cos\varphi)^\alpha
|\langle t,\omega\rangle|^\alpha\). Since the integrand is even in
\(\varphi\), we obtain
\begin{equation}\label{eq:psi-coordinates}
    \Psi_\theta(t)
    =
    c_d\int_{S^{N-1}}|\langle t,\omega\rangle|^\alpha
    \int_0^{b_\theta(\omega)}
    (\cos\varphi)^{\alpha+N-1}\dd\varphi\,\dd\omega .
\end{equation}

\step{Step 3: two-sided estimate.}
For \(0<\theta\leq\pi/4\), we have \(\tan\theta\asymp\theta\) and
\(\tan\theta\,|\omega_N|\leq1\), and hence
\(b_\theta(\omega)\asymp\theta|\omega_N|\). On
\([0,b_\theta(\omega)]\subset[0,\pi/4]\),
\((\cos\varphi)^{\alpha+N-1}\asymp1\). It follows from
\eqref{eq:psi-coordinates} that
\[
    \Psi_\theta(t)
    \asymp
    \theta\,J(t),
    \qquad
    J(t)\eq
    \int_{S^{N-1}}
    |\langle t,\omega\rangle|^\alpha|\omega_N|\dd\omega .
\]
By the dominated convergence theorem, \(J\) is continuous on \(S^{N-1}\),
and it is strictly positive because the set of \(\omega\) for which both
\(\langle t,\omega\rangle\neq0\) and \(\omega_N\neq0\) has positive measure.
By compactness of \(S^{N-1}\), there exist \(0<m\leq M<\infty\) such that
\(m\leq J(t)\leq M\) for \(\|t\|=1\). Homogeneity of \(\Psi_\theta\) of
degree \(\alpha\) gives~\eqref{eq:psi-comparison} for all
\(t\in e_d^\perp\).
\qed

\begin{remark}[the case \(d=2\)]\label{rem:d2}
We spell out the endpoint case of Lemma~\ref{lem:tangent}. When \(d=2\), we have
\(N=1\), so \(S^{N-1}=S^0=\{-1,+1\}\) is a two-point compact space with
counting measure, and the integral \(J(t)\) above equals \(2|t|^\alpha\). All
subsequent arguments remain unchanged: the objects that may initially appear
\enquote{degenerate} are perfectly well defined. The same applies to
Lemma~\ref{lem:arrangement}, which for \(N=1\) says that \(m\) points divide
an interval into at most \(2m+1\) parts.
\end{remark}

\subsection{Number of sign strata}\label{app:arrangement}

Proof of Lemma~\ref{lem:arrangement}. Let \(L_j=\{\ell_j=0\}\), and let
\(\mathcal F\) be the set of all nonempty \emph{faces} of the arrangement
\(L_1,\ldots,L_m\) in \(\R^N\), that is, the nonempty intersections
\[
    \bigcap_{j=1}^{m}S_j,
    \qquad
    S_j\in\bigl\{\{\ell_j<0\},\ \{\ell_j=0\},\ \{\ell_j>0\}\bigr\}.
\]

\step{Step 1: no more strata than faces.}
Each sign stratum is an intersection of sets of the form \(\{\ell_j>0\}\)
and \(\{\ell_j\leq0\}\), and is therefore convex; distinct strata are
disjoint. Every stratum is a union of faces from \(\mathcal F\), and every
nonempty stratum contains at least one face. Hence the number of nonempty
strata is at most \(\#\mathcal F\). Intersecting with the convex set
\(\mathcal D\) can only reduce their number.

\step{Step 2: estimating the number of faces.}
Every \(k\)-dimensional face lies in the intersection of \(N-k\) linearly
independent hyperplanes \(L_j\); there are at most
\(\binom{m}{N-k}\) ways to choose such an intersection. Within the resulting
\(k\)-dimensional affine subspace, the remaining hyperplanes cut out at most
\(\sum_{i=0}^{k}\binom{m}{i}\) full-dimensional regions. Summing over
\(0\leq k\leq N\), we obtain
\[
    \#\mathcal F
    \leq
    \sum_{k=0}^{N}\binom{m}{N-k}\sum_{i=0}^{k}\binom{m}{i}
    \leq
    C_N (m+1)^N ,
\]
because every term is at most
\(C\,m^{N-k}\cdot m^k=C\,m^N\), and the number of terms depends only on
\(N\). Degeneracies, such as coincident or parallel hyperplanes, can only
decrease the number of faces.
\qed

\section{Standard facts used in the proof}\label{app:standard}

\begin{proposition}[{Straszewicz~\nbcite{straszewicz1935}; see
\nbcite[Theorem~1.4.7]{schneider2014}}]\label{prop:straszewicz}
For every compact convex set \(K\subset\R^d\), the set of exposed points is
dense in the set of extreme points:
\(\ext K\subset\overline{\expo K}\).
\end{proposition}

\begin{proposition}[{Frostman's energy criterion; see
\nbcite[Theorem~4.13(a)]{falconer2014}}]\label{prop:frostman}
Let \(A\subset\R^d\), \(s>0\), and let \(\mu\) be a finite nonzero Borel
measure supported on \(A\), that is,
\[
    \mu(\R^d\setminus A)=0.
\]
If
\[
    I_s(\mu)
    \eq
    \iint_{\R^d\times\R^d}
    \|x-y\|^{-s}\,\mu(\dd x)\,\mu(\dd y)
    <\infty,
\]
then \(\mathcal H^s(A)>0\); in particular, \(\dimH A\geq s\).
\end{proposition}

\begin{proposition}[density of a one-dimensional projection]
\label{prop:belt-density}
Let \(\sigma\) be the uniform probability measure on \(\Sph\), with
\(d\geq2\). Then~\eqref{eq:belt-general} holds for every \(u\in\Sph\) and
\(0<r<1\).
\end{proposition}

\begin{proof}
By invariance of \(\sigma\), it suffices to take \(u=e_d\). For \(d\geq3\),
the last coordinate of a uniform point on the sphere has density
\(c_d(1-y^2)^{(d-3)/2}\) on \([-1,1]\), which is bounded near zero; hence
\(\sigma\{|x_d|\leq r\}\leq C_d r\). For \(d=2\), the set
\(\{|x_2|\leq r\}\) consists of two arcs of length comparable to \(r\).
\end{proof}

\medskip

The next two statements, formulated in
Remarks~\ref{rem:exposed-extreme} and~\ref{rem:extreme-not-exposed}, are
standard facts of convex geometry rather than new results of this paper. They
are included for completeness.

\begin{proposition}[{exposed points are extreme; see
\nbcite[Proposition~2.6.2]{tropp2018}}]\label{prop:exposed-implies-extreme}
The assertion of Remark~\ref{rem:exposed-extreme} holds: every compact convex
set \(K\subset\R^d\) satisfies
\[
    \expo K\subset\ext K.
\]
\end{proposition}

\begin{proof}
Let \(x\in\expo K\), and suppose that the functional
\(z\mapsto\langle u,z\rangle\) attains its maximum over \(K\) only at \(x\).
If \(x=ty+(1-t)z\), where \(y,z\in K\) and \(t\in(0,1)\), then
\[
    \langle u,x\rangle
    =
    t\langle u,y\rangle+(1-t)\langle u,z\rangle.
\]
Both terms on the right are at most \(\langle u,x\rangle\), so equality is
possible only if
\(\langle u,y\rangle=\langle u,z\rangle=\langle u,x\rangle\). Uniqueness of
the maximizer gives \(y=z=x\), and thus \(x\in\ext K\).
\end{proof}

\begin{proposition}[{the reverse inclusion fails; see
\nbcite[Fig.~2.10]{tropp2018}}]\label{prop:extreme-not-exposed}
The assertion of Remark~\ref{rem:extreme-not-exposed} holds. For example, for
the compact convex set
\[
    K
    =
    \bigl([-1,1]\times[-2,0]\bigr)
    \cup
    \bigl\{(x,y)\in\R^2:x^2+y^2\leq1,\ y\geq0\bigr\},
\]
the points \(p_\pm=(\pm1,0)\) belong to
\(\ext K\setminus\expo K\).
\end{proposition}

\begin{proof}
The compact set \(K\) is obtained by attaching the upper half-disk to the top
side of the square (Fig.~\ref{fig:nonexposed}). If
\(p_\pm=ta+(1-t)b\), where \(a,b\in K\) and \(t\in(0,1)\), then the
first-coordinate condition followed by the second-coordinate condition forces
\(a=b=p_\pm\); hence
\(p_\pm\) are extreme. The unique supporting line at \(p_\pm\), namely the
common tangent \(x=\pm1\) to the arc and the side, intersects \(K\) along the
entire segment \(\{\pm1\}\times[-2,0]\). Thus \(p_\pm\) are not exposed.
\end{proof}

\begin{figure}[ht]
\centering
\begin{tikzpicture}[x=1.25cm,y=1.25cm]
\path[fill=blue!7]
    (-1,-2)--(1,-2)--(1,0)
    arc[start angle=0,end angle=180,radius=1]
    --(-1,-2)--cycle;
\draw[blue!65!black,thick]
    (-1,-2)--(1,-2)--(1,0)
    arc[start angle=0,end angle=180,radius=1]
    --cycle;
\draw[blue!65!black,densely dashed] (-1,0)--(1,0);
\draw[gray,dashed] (-1,-2.25)--(-1,1.25);
\draw[gray,dashed] (1,-2.25)--(1,1.25);
\fill[red!75!black] (-1,0) circle (2.2pt);
\fill[red!75!black] (1,0) circle (2.2pt);
\node[red!75!black,above left] at (-1,0) {\(p_-\)};
\node[red!75!black,above right] at (1,0) {\(p_+\)};
\node[font=\small] at (0,-1) {square};
\node[font=\small] at (0,0.55) {half-disk};
\end{tikzpicture}
\caption{Counterexample to the reverse inclusion: the red points \(p_-\) and
\(p_+\) are extreme, but the tangent lines \(x=\pm1\) support entire sides of
the square.}\label{fig:nonexposed}
\end{figure}
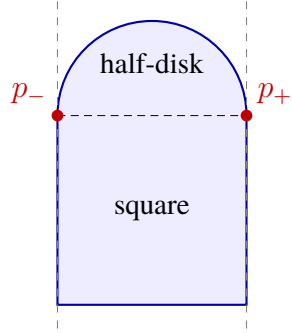

\begin{remark}[analytic sets and universal measurability]
\label{rem:analytic-universal}
We explain the step used in the proof of Theorem~\ref{thm:lower}. Here
\emph{analytic} has no connection with analytic functions: a subset \(A\) of
a Polish space \(E\) is analytic if \(A=f(P)\) for some Polish space \(P\) and
some continuous map \(f\colon P\to E\). Equivalently, \(A\) is the projection
of a Borel subset of a product of Polish spaces; in particular, the Borel
image of a Borel set is analytic
\nbcite[Definition~14.1 and Proposition~14.4]{kechris1995}.

A set \(A\subset E\) is \emph{universally measurable} if it belongs to the
completion of the Borel \(\sigma\)-algebra with respect to every finite Borel
measure \(\nu\) on \(E\). Equivalently, there exist Borel sets
\(B_\nu,C_\nu\subset E\) such that
\[
    B_\nu\subset A\subset C_\nu,
    \qquad
    \nu(C_\nu\setminus B_\nu)=0.
\]
Every analytic subset of a Polish space is universally measurable
\nbcite[Theorem~21.10]{kechris1995}; see also
\nbcite[Section~6.7]{bogachev2007}.

In Theorem~\ref{thm:lower}, the spaces \(\Sph\) and \(\R^d\) are Polish,
\(G\cap U\) is Borel, and \(X\) is a Borel map. Thus
\(A=X(G\cap U)\) is analytic and universally measurable. After completing
the measures \(\lambda\) and \(\mu=X_{\#}\lambda\), the identity defining the
pushforward extends to \(A\) and its complement, and hence
\[
    \mu(\R^d\setminus A)
    =
    \lambda\bigl\{u\in U:X(u)\notin A\bigr\}
    \leq
    \lambda(U\setminus G)
    =
    0.
\]
Consequently, although \(A\) need not be Borel, it is \(\mu\)-measurable and
\(\mu\) is supported on \(A\).
\end{remark}

\clearpage
\printbibliography[title={References}]

@article{gine1985,
  author       = {Gin{\'e}, Evarist and Hahn, Marjorie G.},
  title        = {Characterization and Domains of Attraction of {$p$}-Stable Random Compact Sets},
  journaltitle = {The Annals of Probability},
  volume       = {13},
  number       = {2},
  pages        = {447--468},
  year         = {1985}
}

@article{davydov2000,
  author       = {Davydov, Youri and Paulauskas, Vygantas and Ra{\v{c}}kauskas, Alfredas},
  title        = {More on {$p$}-Stable Convex Sets in Banach Spaces},
  journaltitle = {Journal of Theoretical Probability},
  volume       = {13},
  number       = {1},
  pages        = {39--64},
  year         = {2000}
}

@incollection{davydov2004,
  author    = {Davydov, Youri and Paulauskas, Vygantas},
  title     = {Recent Results on {$p$}-Stable Convex Compact Sets with Applications},
  booktitle = {Asymptotic Methods in Stochastics: Festschrift for Mikl{\'o}s Cs{\"o}rg{\H{o}}},
  editor    = {Horv{\'a}th, Lajos and Szyszkowicz, Barbara},
  series    = {Fields Institute Communications},
  volume    = {44},
  pages     = {127--139},
  publisher = {American Mathematical Society},
  location  = {Providence, RI},
  year      = {2004}
}

@article{davydov2008,
  author       = {Davydov, Youri and Molchanov, Ilya and Zuyev, Sergei},
  title        = {Strictly Stable Distributions on Convex Cones},
  journaltitle = {Electronic Journal of Probability},
  volume       = {13},
  number       = {11},
  pages        = {259--321},
  year         = {2008}
}

@article{davydov2019,
  author       = {Davydov, Youri and Paulauskas, Vygantas},
  title        = {Remarks on Random Countable Stable Zonotopes},
  journaltitle = {Statistics \& Probability Letters},
  volume       = {153},
  pages        = {187--191},
  year         = {2019}
}

@article{kukushkin2025,
  author       = {Kukushkin, Maksim A.},
  title        = {Geometrical Characteristics of Random Countable Zonotopes},
  journaltitle = {Zapiski Nauchnykh Seminarov POMI},
  volume       = {544},
  pages        = {185--210},
  year         = {2025},
  langid       = {russian}
}

@article{straszewicz1935,
  author       = {Straszewicz, Stefan},
  title        = {{\"U}ber exponierte Punkte abgeschlossener Punktmengen},
  journaltitle = {Fundamenta Mathematicae},
  volume       = {24},
  number       = {1},
  pages        = {139--143},
  year         = {1935},
  langid       = {german}
}

@article{taylor1966,
  author       = {Taylor, S. James and Wendel, James G.},
  title        = {The Exact {H}ausdorff Measure of the Zero Set of a Stable Process},
  journaltitle = {Zeitschrift f{\"u}r Wahrscheinlichkeitstheorie und Verwandte Gebiete},
  volume       = {6},
  number       = {2},
  pages        = {170--180},
  year         = {1966}
}

@book{molchanov2005,
  author    = {Molchanov, Ilya},
  title     = {Theory of Random Sets},
  series    = {Probability and Its Applications},
  publisher = {Springer},
  location  = {London},
  year      = {2005}
}

@book{schneider2014,
  author    = {Schneider, Rolf},
  title     = {Convex Bodies: The {B}runn--{M}inkowski Theory},
  edition   = {2},
  publisher = {Cambridge University Press},
  location  = {Cambridge},
  year      = {2014}
}

@report{tropp2018,
  author      = {Tropp, Joel A.},
  title       = {{ACM} 204: Lectures on Convex Geometry},
  type        = {Caltech CMS Lecture Notes},
  number      = {2019-02},
  institution = {California Institute of Technology},
  location    = {Pasadena, CA},
  year        = {2018}
}

@book{falconer2014,
  author    = {Falconer, Kenneth},
  title     = {Fractal Geometry: Mathematical Foundations and Applications},
  edition   = {3},
  publisher = {Wiley},
  location  = {Chichester},
  year      = {2014}
}

@book{mattila1995,
  author    = {Mattila, Pertti},
  title     = {Geometry of Sets and Measures in Euclidean Spaces},
  publisher = {Cambridge University Press},
  location  = {Cambridge},
  year      = {1995}
}

@book{samorodnitsky1994,
  author    = {Samorodnitsky, Gennady and Taqqu, Murad S.},
  title     = {Stable Non-Gaussian Random Processes: Stochastic Models with Infinite Variance},
  publisher = {Chapman \& Hall},
  location  = {New York},
  year      = {1994}
}

@book{bertoin1996,
  author    = {Bertoin, Jean},
  title     = {L{\'e}vy Processes},
  publisher = {Cambridge University Press},
  location  = {Cambridge},
  year      = {1996}
}

@book{kechris1995,
  author    = {Kechris, Alexander S.},
  title     = {Classical Descriptive Set Theory},
  series    = {Graduate Texts in Mathematics},
  volume    = {156},
  publisher = {Springer},
  location  = {New York},
  year      = {1995}
}

@book{bogachev2007,
  author    = {Bogachev, Vladimir I.},
  title     = {Measure Theory},
  volume    = {2},
  publisher = {Springer},
  location  = {Berlin},
  year      = {2007}
}

\end{document}